\documentclass[12pt]{amsart}
\usepackage{amssymb,amsfonts}
\usepackage{amsmath,amsthm,amsxtra}

\newtheorem{theorem}{Theorem}
\newtheorem{lemma}[theorem]{Lemma}
\numberwithin{theorem}{section}
\newtheorem{proposition}[theorem]{Proposition}
\theoremstyle{definition}
\newtheorem{definition}[theorem]{Definition}
\newtheorem{remark}[theorem]{Remark}

\numberwithin{equation}{section} 

\newcommand{\al}{\alpha}
\newcommand{\la}{\lambda}
\newcommand{\La}{\Lambda}
\newcommand{\CC}{\mathbb{C}}
\newcommand{\DD}{\mathbf{D}}
\newcommand{\RR}{\mathbb{R}}

\newcommand{\fa}{\mathfrak{a}}
\newcommand{\fg}{\mathfrak{g}}
\newcommand{\fk}{\mathfrak{k}}

\newcommand{\fp}{\mathfrak{p}}

\newcommand{\half}{\frac{1}{2}}
\newcommand{\thalf}{\tfrac{1}{2}}

\newcommand{\Ad}{\mathrm{Ad} \,}

\renewcommand{\Re}{\mathrm{Re} \,}
\renewcommand{\Im}{\mathrm{Im} \,}

\title[Spherical functions]{Spherical functions on Riemannian symmetric spaces}
\author{Sigurdur Helgason}
\dedicatory{Dedicated to Professor Gestur \'{O}lafsson on his 65th birthday.}


\begin{document}
\maketitle

\section{Introduction}
Let $ X = G/K $ be a symmetric space where $ G $ is a connected noncompact semisimple Lie group with finite center and $ K $ a maximal compact subgroup. Let $ \mathbf{D}(X)  $ denote the algebra of $ G $-invariant differential operators on $ X. $ Let $ \delta $ be a unitary irreducible representation of $ K $ on a vector space  $ V_\delta.  $

\begin{definition}
\label{def1.1}
A spherical function of $ K $ type $ \delta $ is a $ C^\infty $ function $ \Phi : X \rightarrow \mathrm{Hom}\, (V_\delta, V_\delta) $ satisfying the following conditions.
\begin{equation}
\label{1.1}
 \Phi  \text{ is an eigenfunction of each }  D \in \mathbf{D}(X) 
\end{equation}
\begin{equation}
\label{1.2}
 \Phi (k \cdot x) = \delta (k) \Phi (x) \quad k \in K, x \in X. 
\end{equation}

On the right we have multiplication in $ \text{Hom} \, (V_\delta, V_\delta). $
\end{definition}

\begin{remark}
This definition has some similarity with those of Godement \cite{G52} and Harish-Chandra \cite{HC72}. However, the first one is modeled after invariance under $ g \rightarrow kgk^{-1}; $ the second deals with bi-invariant differential operators on $ G $ and a double representation of $ K \times K. $ Our definition by conditions \eqref{1.1} and \eqref{1.2} is thus rather different. 

It is also unrelated to a definition of a spherical function by Tirao \cite{T76} as a function characterized by the function equation (3.3) for zonal spherical functions but with the measure $ dk $ replaced by $ \chi (k) \, dk $ where $ \chi $ is the character of a representation of $ K. $
\end{remark}

Our definition (\ref{def1.1}) stresses spherical functions as functions on $ X $ rather than as $ K $-right invariant functions on $ G. $ I remark that Harish-Chandra's major papers \cite{HC58} on zonal spherical functions do not actually mention the space $ X = G/K $ nor the algebra $ \DD (X). $

This paper deals with some simple results about the functions satisfying \eqref{1.1} and \eqref{1.2}, namely new integral formulas, new results about behavior at infinity and some facts about the related $ C_\sigma $ functions. 

Appreciation to G. Olafsson, A. Pasquale, D. Vogan, and J. A. Wolf for helpful comments. I am also grateful to the referee for several corrective suggestions. 

\section{Notation and background}
As usual $ \mathbf{R}, \mathbf{C}, $ and $ \mathbf{Z} $ denote the sets of real numbers, the complex numbers and integers, respectively. If $ c = a + ib, \ a, b \in \mathbf{R} $ we write $ a = \Re c, b = \Im c  $ and $ \bar{c} = a - ib.  $ If $ L $ is a Lie group with Lie algebra $ \mathfrak{l}, \exp: \mathfrak{l} \rightarrow L $ denotes the exponential mapping and ad (resp. Ad) the adjoint representation of $ \mathfrak{l} $ (resp. $ L $).

Going back to \S 1, our group $ G $ has a Lie algebra $ \fg $ with Cartan decomposition $ \fg = \fk + \fp $ where $ \fk $ is the Lie algebra of $ K $ and $ \fp $ is the orthocomplement of $ \fk $ relative to the Killing form $ B $ of $ \fg. $ We fix a maximal abelian subspace $ \fa $ of $ \fp $ and fix a Weyl chamber $ \fa^+ \subset \fa. $ All such choices are conjugate under $ \Ad (K). $ The choice of $ \fa^+  $ induces Iwasawa decompositions $ G = NAK $ and $ G = KAN $ where $ A = \exp \fa $ and $ N $  is nilpotent. In these decompositions we write $ g = n \exp A(g)k, g = k_1 \exp H(g) n_1 $ where $ A(g) $ and $ H(g) $ are uniquely determined in $ \fa $ and $ A(g) = - H(g^{-1}). $


If $ M $ is the centralizer of $ A $ in $ K $ the ``vector valued'' inner product $ A(gK, kM) = A(k^{-1}g) $ is well defined and considered analog to the Euclidean $ (x,w) $ the distance from 0 to the hyperplane through $ x $ with unit normal $ w $. We put $ B = K/M. $ If $ \widehat{K} $ denotes the set of irreducible representations of $ K, $ condition \eqref{1.2} implies that $ \delta(M) $ has a common fixed point; we denote by $ \widehat{K}_M $ the set of these $ \delta.  $

Let $ V_\delta $ denote the space on which $ \delta $ operates and $ V^M_\delta $ the subspace of fixed points under $ \delta(M). $ We also use the notation $ \mathcal{E}(X) $ (resp. $ \mathcal{D}(X) $) for the space of $ C^\infty $ functions on $ X $ (resp. those of compact support).

We denote by $ \pi $ the natural map of $ G $ on $ G/K $ and put $ \tilde{f} = f \circ \pi $ for a function $ f $ on $ G/K. $ We also denote by $ \circ $ the coset $ eK. $

Let $ \fa^* $ (resp. $ \fa^*_\CC $) be the space of $ \mathbf{R} $-linear maps of $ \fa $ into $ \mathbf{R} $ (resp. $ \mathbf{C} $). In the bijection of $ \fa $ with $ \fa^* $ via the Killing form of $ \fg $ let $ \fa^*_+ $ correspond to $ \fa^+ $. Let $ S(\fa) $ denote the symmetric algebra over $ \fa $ and $ I(\fa) $ the subspace of $ p \in S(\fa) $ invariant under the Weyl group $ W.  $ We put $ \bar{N} = \theta N $ if $ \theta  $ is the Cartan involution. Let $ M' $ be the normalizer of $ A $ in $ K. $ If $ \sigma \in W $ and $ m_\sigma $ representing $ \sigma $ in $ M' $ we put 
\[ \bar{N}_\sigma = \bar{N} \cap m^{-1}_\sigma N m_\sigma.  \]
This group appears later.

\section{Zonal spherical functions}
%


A \textbf{zonal spherical function} $ \phi $ on $ G $ is a $ C^\infty $ function on $ G $ satisfying 
\begin{equation}
\label{3.1}
 \phi  \mbox{ is an eigenfunction of each } D \in \mathbf{D}_K(G). 
\end{equation}
\begin{equation}
\label{3.2}
 \phi \mbox{ is bi-invariant under } K, \ \phi (e) = 1. 
\end{equation}

Here $ \mathbf{D}_K (G) $ is the algebra of differential operators on $ G $ which are left invariant under $ G $ and right invariant under $ K. $

Properties \eqref{3.1} and \eqref{3.2} are well known to be equivalent to 
\begin{equation}
\label{3.3}
 \int_{K} \phi (xky) dk = \phi (x) \phi (y). 
\end{equation}

%

The zonal spherical functions are all given by Harish-Chandra's formula (\cite{HC54})
\begin{equation}
\label{3.4}
\phi (g) = \int_K e^{(i \la - \rho)(H(gk))} dk
\end{equation}
for some $ \la \in \fa^*_C. $ Here $ \rho $ is half the sum of the positive restricted roots with multiplicity. Writing $ \phi = \phi_\la $ this function has from \cite{HC58} an expansion 

\begin{equation}
\label{3.5}
\phi_\la (a) = \sum_{s \in W} \mathbf{c} (s \la) e^{(is \la - \rho)(\log a)} \sum_{\mu \in \La} \Gamma (s \la) e^{-\mu (\log a)}
\end{equation}
for $ a $ in Weyl chamber $ \exp(\fa^+), \ W $ the Weyl group, $ \La  $ the lattice 
\[ \La = \{ m_1 \al_1 + \cdots + m_\ell \al_\ell \ | \ m_i \in \mathbf{Z}^+ \}, \]
the $ \al_1, \ldots, \al_\ell $ being the simple restricted roots. This is where the remarkable $ \mathbf{c} $ function first appears. 

The $ \Gamma $ are rational functions on $ \fa^*_\mathbf{c} $ and $ \mathbf{c} $ is a meromorphic function on $ \fa^*_\mathbf{c} $ given by Harish-Chandra \cite{HC58} as the integral
\begin{equation}
\label{3.6}
\mathbf{c} (\la) = \int_{\bar{N}} e^{-(i\la + \rho) (H (\bar{n}))} d \bar{n}.
\end{equation}

Through the work of Harish-Chandra \cite{HC58}, Bhanu-Murthy \cite{BM60} and Gindikin-Karpelevic \cite{GK62} the $ \mathbf{c} $ function is given by
\begin{equation}
\label{3.7}
\mathbf{c} (\la) = \prod_{\al \in \sum_{0}^{+}} \frac{2^{-\langle i \la - \rho, \al_0 \rangle} \Gamma \left( \frac{1}{2} (m_\al + m_{2 \al} + 1)\right) \Gamma (\langle i \la, \al_0 \rangle )}{\Gamma \left( \half (\half m_\al + 1 + \langle i\la, \al_0 \rangle ) \right) \Gamma \left( \half (\half m_\al + m_{2 \al} + \langle i \la, \al_0 \rangle)  \right) }.
\end{equation}

Here $ \sum_{0}^{+} $ denotes the set of positive, indivisible roots, $ m_\al $ the multiplicity of $ \al $ and $ \al_0 = \al / \langle \al, \al \rangle. $

In Harish-Chandra's work,  $ | \mathbf{c} (\la)|^{-2} $ served as the dual measure for the \textbf{spherical transform} on $ G. $ However, formula \eqref{3.7} has many other interesting features. See e.g. \cite{H00}.


\section{The spaces $ X = G/K $ and its Dual $ \Xi = G/MN $}

As proved in \cite{H62}, p. 439 and \cite{H70}, p. 94 the modified integrand in \eqref{3.4}, that is the function 
\begin{equation}
\label{a4.1}
gK \rightarrow e^{(i \la + \rho) (A(k^{-1}g))}, 
\end{equation}
is for each $ k $ an eigenfunction of $ \DD (X) $ and the eigenvalue is $ \Gamma (D) (i \la) $ where $ D \in \DD(X) $ and $ \Gamma (D) \in I(\fa). $ (This is related to Lemma 3 in \cite{HC58}, I) but not contained in it). The map $ \Gamma $ is spelled out in \cite{H84}, II, Theorem 5.18.

This led in \cite{H65} to the definition of a \textbf{Fourier transform} $ f \rightarrow \tilde{f}  $ for a function $ f $ on $ X, $
\begin{equation}
\label{a4.2}
\tilde{f} (\la, b) = \int_X f(x) e^{(-i\la + \rho)(A(x,b))} \, dx, \ b \in B, \la \in \fa^*_\mathbf{c}
\end{equation}
in analogy with the polar coordinate expression
\[ \tilde{F}(\la w) = \int_X F(x)e^{-i \la (x,w)}  \, dx, \ |w| = 1 \]
for the Fourier transform on $ \RR^n. $ Here $ dx $ denotes the volume element in both cases. 

In addition we consider the \textbf{Poisson transform}

\begin{equation}
\label{a4.3}
(\mathcal{P}_\la F) (x) = \int_B e^{(i\la + \rho) (A(x,b))} F(b) \, db, \ F \text{ a function on } B.
\end{equation}

These transforms are intimately related to the $ \mathbf{c} $-function. By \cite{H70}, p. 120, the map $ f \rightarrow \tilde{f} $ is an isometry of $ L^2(X) $ onto $ L^2 (\fa^*_+ \times  B; | \mathbf{c} (\la) |^{-2} \, d \la \, db). $ Secondly, $ \mathcal{P}_\la $ is related to the \textbf{denominator} $ \Gamma^+_X (\la) $ in \eqref{3.7}, called the Gamma function of $ X. $

This $ \mathcal{P}_\la $ is closely related to the dual Radon transform $ \phi \rightarrow \overset{\vee}{\phi} $ from $ \Xi $, the space of horocycles in $ X, $ to $ X $ which to a function $ \phi $ on $ \Xi $ associates $ \overset{\vee}{\phi} (x), $ the average of $ \phi $ over horicycles $ \xi \in \Xi $ passing  through $ x \in X  $ (\cite{H08}, p. 103).

The element $ \la \in \fa^*_c $ is said to be \textbf{simple} if $ \mathcal{P}_\la $ is injective. The connection with the denominator in \eqref{3.7} is (\cite{H76}):

\begin{theorem}
\label{th3.1}
$ \la $ is non-simple if and only if $ \Gamma^+_X (\la)^{-1} = 0. $
\end{theorem}

On the other hand, the \textbf{numerator} in \eqref{3.7} is connected with analysis on the dual of $ X, $ that is, the space $ \Xi = G/MN $ of horocycles in $ X $. The counterpart to the zonal spherical functions on $ X $ would be the $ MN $-invariant  eigendistributions of $ \mathbf{D}(G/MN), $ the algebra of $ G $-invariant differential operators on $ \Xi.  $ These ``\textbf{conical distributions}'' have a construction and theory in \cite{H70}, using the numerator of \eqref{3.7}. 


While the set of zonal spherical functions is parametrized by $ \fa^*_c / W $ via \eqref{3.4}, the set of conical distributions turned out to be essentially parametrized by $ \fa^*_c \times W. $ 

In the proof of \eqref{3.7} the following partial $ \mathbf{c} $-function $ \mathbf{c}_\sigma $  enters for each $ \sigma \in W.  $ In analogy with \eqref{3.6}, it is defined by 
\[ \mathbf{c}_\sigma (\la) = \int_{\bar{N}_\sigma} e^{-(i \la + \rho) (H(\bar{n}))} d \bar{n}, \quad \bar{N}_\sigma = \bar{N} \cap N^{\sigma^{-1}},  \]
for a suitable normalization of $ d \bar{n} $ on $ \bar{N}_\sigma $. It has a formula generalizing \eqref{3.7},
\[ \mathbf{c}_\sigma (\la) = \prod_{\al \in \sum_{0}^{+} \bigcap \sigma^{-1} \sum_{0}^{-}} \mathbf{c}_\al (\la_\al), \quad \la_\al = \la | \fa_\al,  \]
where $ \mathbf{c}_\al $ is the $ c $-function for the group $ G_\al \subset G $ whose Lie algebra is the subalgebra of $ \fg $ generated by $ \fg_\al  $ and $ \fg_{-\al}. $

The numerous beautiful features of \eqref{3.7} are the cause of the title in \cite{H00}.

\section{Global descriptions of eigenspaces}
%

Let $ V $ be a finite dimensional vector space. A function $ \Phi: X \rightarrow V $ satisfying \eqref{1.1} is called a \textit{joint eigenfuction} of $ \DD (X). $ The eigenvalue $ \chi (D) $ in $ D\Phi = \chi (D) \Phi $ is a homomorphism of $ I(\fa) $ into $ \mathbf{C} $ and thus has the form $ \chi (D) = \Gamma (D) (i \la) $ for some $ \la \in \fa^*. $ Corresponding eigenspaces are \textit{joint eigenspaces}. The scalar version of Proposition \ref{prop5.1} is from \cite{H62}, X, \S 7. We shall use \cite{H84}, mainly Ch. II, \S 4 and Ch.IV, \S 2. 

\begin{proposition}
\label{prop5.1}
The joint eigenfunctions $ \Phi: X \rightarrow V $ of $ \DD (X) $ are the continuous functions $ \Phi $ satisfying. 
\begin{equation}\label{4.1}
\int_K \Phi (xky \cdot \circ) dk = \Phi (x \cdot \circ) \phi_\la (y \cdot \circ), \quad x,y \in G
\end{equation}
for some $ \la \in \fa^*_c. $
\end{proposition}

\begin{proof}

First assume \eqref{4.1}. Integrating \eqref{4.1} against a test function $ f  $ in $ y, \ \Phi(x) $ can be written in terms of derivative of $ f $ in $ x $ so $ \Phi $ is smooth. Then applying $ D_y $ to \eqref{4.1} and putting $ y = e $ we find
\begin{equation}
\label{4.2}
D\Phi = \Gamma (D) (i \la) \Phi
\end{equation}
so $ \Phi  $ is an eigenfunction. For the converse assume \eqref{4.2} and put 
\[ \Psi_x (y) = \int_K \Phi (xky) \, dk \]
and note that 
\begin{equation}
\label{4.3}
D \Psi_x = \Gamma (D) (i \la) \Psi_x. 
\end{equation}
Let $ \mathbf{D} (G) $ denote the algebra of left invariant differential operators on $ G $ and as before $ \mathbf{D}_K (G) $ the subalgebra of those which are also right $ K $-invariant. For $ D \in \mathbf{D}(G) $ let 
\begin{equation}
\label{4.4}
D_0 = \int_K \Ad (k) D \, dk
\end{equation}
and recall that $ D \rightarrow D_0 $ maps $ \mathbf{D}(G) $ onto $ \mathbf{D}_K(G). $ Also consider the map $ \mu $ given by 
\begin{equation}
\label{4.5}
(\mu (u)f)^\sim = u \tilde{f} \qquad u \in \mathbf{D}_K (G), \ \tilde{f} = f \circ \pi
\end{equation}
which maps $ \mathbf{D}_K (G) $ onto $ \mathbf{D} (G/K) $ (\cite{H84}, Ch. II, \S 4).

Among $ D $ in \eqref{4.2} is the Laplace-Beltrami  operator so $ \Phi $ is an analytic function. If $ F \in C^\infty (G) $ is bi-invariant under $ K $ we have by \cite{H84} (3), p. 400,
\begin{equation}
\label{4.6}
(D_0F) (e) = (DF)(e).
\end{equation}
This applies both to $ F(y) = \Psi_x (y \cdot \circ) $ and $ F(y) = \phi _\la (y \cdot \circ)$ and to 
\[ f(y \cdot \circ) = \phi _\la (e) \Psi_x (y \cdot \circ) - \Psi_x (e) \phi _\la (y \cdot \circ). \]
For $ D \in \DD (G) $ arbitrary we take $ u = D_0 $ in \eqref{4.5}. Then using \eqref{4.6}, $ (D \tilde{f})(e) = 0. $ Since $ f $ is analytic and $ f(\circ) = 0 $ we have $ f \equiv 0 $ which is formula \eqref{5.1}.
\end{proof}

With $ V $ as before we consider joint eigenfunctions for the algebra $ \DD (\Xi) = \DD (G/MN). $ If $ \DD (A) $ denotes the left invariant differential operators on $ A $ each $ U \in \DD(A)  $ induces an $ D_U \in \DD (\Xi) $ by
\[ (D_U \phi) (kaMN) = U_a (\phi (kaMN)) \]
and $ U \rightarrow D_U $ is an isomorphism of $ \DD (A) $ onto $ D(\Xi) $ (\cite{H70}, I, \S 2). We denote by $ \hat{\Gamma} $ its inverse. Let $ \mathcal{E}_\la (\Xi) $ denote the joint eigenspace
\[ \mathcal{E}_\la (\Xi) = \left\{ \Psi \in \mathcal{E} (\Xi): D \Psi = \hat{\Gamma} (D) (i \la - \rho)  \Psi  \right\}. \]

\begin{proposition}
\label{prop5.2}
The joint eigenfunctions $ \Psi $ of $ \Xi $ into $ V $ are the smooth functions satisfying
\begin{equation}\label{5.2}
\Psi (gaMN) = \Psi (gMN)e^{(i\la- \rho) (\log a)}, \quad a \in A,
\end{equation}
for some $ \la \in \fa^*_c. $
\end{proposition}

For proof see \cite{H08}, II, \S 2, also for distributions on $ \Xi. $

\section{Spherical functions of a given $ K $-type}
Let $ \Phi $ be a spherical function of type $ \delta $ as in \S 1. Thus we take $ V = \text{Hom} \, (V_\delta, V_\delta) $ in \S5 and assume \eqref{1.2}. As mentioned in \S 5 there exists a  $ \la \in \mathfrak{a}^*_c $ such that
\begin{equation}
\label{5.1}
D \Phi = \Gamma (D) (i \la) \Phi, \quad D \in \mathbf{D}(X).
\end{equation}

This $ \la $ is unique up to conjugacy by $ W $ and since each $ \la \in \fa^*_\mathbf{c} $ is $ W $ conjugate to one which is simple we can take $ \la \in \fa^*_\mathbf{c} $ to be simple. By definition, $ \mathcal{P}_\la $ is injective.

Let $ v_1, \ldots , v_{d(\delta)} $ be an orthonormal basis of $ V_\delta $ such that $ v_1, \ldots, v_{\ell(\delta)} $ span $ V^M_\delta. $ Then
\begin{equation}
\label{5.2}
\Phi (x) v_j = \sum_{i = 1}^{d(\delta)} \phi _{ij} (x) v_i
\end{equation}
Then condition \eqref{1.2} implies 
\[ \phi _{rj} (k \cdot x) = \sum_{i} \delta_{ri} (k) \phi _{ij} (x), \]
where $ \delta_{ri} (k) $ is the expression of $ \delta(k) $ in the basis $ (v_i).  $  Let $ \mathcal{E}_\la (X) $ denote the space of joint eigenfunctions of $ \mathbf{D} (X) $ with eigenvalues $ \Gamma (D)(i \la). $ Let $ \overset{\vee}{\delta} $ denote the contragredient to $ \delta, \ d(\overset{\vee}{\delta}) = d (\delta) $ its dimension and $ \chi_{\overset{\vee}{\delta}} = \bar{\chi}_\delta $ its character.  Let $ \pi $ denote the representation of $ K $ on $ \mathcal{E}_\la (X) $ given by $ \pi (k) : f(x) \rightarrow f (k^{-1} \cdot x). $ By \cite{H84}, IV, \S 1, the map
\[ d(\overset{\vee}{\delta}) \pi (\bar{\chi}_{\overset{\vee}{\delta}}) \]
is the projection of $ \mathcal{E}_\la (X) $ onto the space $ \mathcal{E}_{\la, \overset{\vee}{\delta}} (X) $ of $ K $-finite elements in $ \mathcal{E}_\la (X)  $ of type $ \overset{\vee}{\delta}. $

\begin{lemma}
Each function $ \phi_{ij} $ in \eqref{6.2} belongs to $ \mathcal{E}_{\la, \overset{\vee}{\delta}} (X) $.
\end{lemma}

\begin{proof}
We have 
\[ \begin{aligned}
d (\overset{\vee}{\delta}) \pi (\bar{\chi}_{\overset{\vee}{\delta}} ) (\phi_{rj}) (x) & = d(\overset{\vee}{\delta}) \int_K \overline{\chi_{\overset{\vee}{\delta}} (k)} \phi_{rj} (k^{-1} \cdot x) dk \\
& = d(\delta) \int_K \overline{\chi_\delta (k)} \sum_{i} \delta_{ri} (k) \phi_{ij} (x) dk \\
\end{aligned} \]
which by Schur's orthogonality relations reduces to $ \phi_{rj} (x).  $ Thus $ \phi_{ij} \in \mathcal{E}_{\la, \overset{\vee}{\delta}} (X) $ for all $ i, j. $
\end{proof}

By \cite{H70}, IV, \S 1 and \cite{H76} \S 7, invoking the Paley-Wiener theorem for \eqref{a4.2}, each $ K $-finite joint eigenfunction of $ \DD (X) $ is the Poisson transform of a $ K $-finite member of $ \mathcal{E} (B). $ We apply this to $ \phi_{ij} (x). $ As is well known the functions 
\[ \langle \delta(k) v_j, v_i \rangle \quad 1 \leq i \leq d(\delta), \ 1 \leq j \leq \ell(\delta) \]
form a basis of $ \mathcal{E}_{\overset{\vee}{\delta}} (B), $ the space of $ K $-finite functions in $ \mathcal{E}(B) $  of type $ \overset{\vee}{\delta}$. The corresponding images under $ \mathcal{P}_\la $
\begin{equation}
\label{5.3}
 \int_{K/M} e^{(i\la + \rho) (A(x, kM))} \langle \delta (k) v_j, v_i \rangle \, dk_M 
 \end{equation}
will by \cite{H73, H76} span the space $ \mathcal{E}_{\la, \overset{\vee}{\delta}} (X)$. Consider the Eisenstein integral
\begin{equation}
\label{5.4}
\Phi_{\la, \delta} (x) = \int_K e^{(i \la + \rho) (A(x, kM))} \delta(k) \, dk, 
\end{equation}
whose matrix entries are given by \eqref{5.3}.

Changing from $ \la  $ to $ s \la (s \in W) $ only changes $ \Phi_{\la, \delta} $ by a factor independent of $ x $ \cite{H73}. This proves following result. 

\begin{theorem}
\label{th5.1}
Each spherical function of type $ \delta $ has each of its matrix entries linear combinations of the functions 
\[ \langle \Phi_{\la, \delta} (x) v_j, v_i \rangle \quad 1 \leq j \leq \ell(\delta), 1 \leq i \leq d(\delta). \]
\end{theorem}

These functions all turn out to be suitable derivative of the zonal spherical functions \cite{H76}.

\begin{theorem}
\label{th5.2}
Fix $ v, w \in V_\delta $ and assume $ -\la $ simple. Then there exists a right invariant differential operator $ D $ on $ G $ such that
\[ \langle v, \Phi_{\la, \delta} (gK), w \rangle = (D\phi _\la) (g). \]
\end{theorem}

This in return implies a series expansion of $ \Phi_{\la, \delta} $ on $ \exp \fa^+,  $ generalizing Harish-Chandra's expansion \eqref{3.5} which introduced the \textbf{c} function.

\begin{theorem}
\label{th5.3}
There exist meromorphic functions $ \mathbf{C}_\sigma (\sigma \in W) $ and rational functions $ \boldsymbol{\Gamma}_\mu (\mu \in \La) $ all with values in $ \mathrm{Hom} \, (V_\delta^M, V_\delta^M) $ such that for $ H \in \fa^+, v \in V^M_\delta, $
\[ \Phi_{\la, \delta} (\exp H \cdot \circ)v = \sum_{\sigma \in W, \mu \in \La} e^{(i \sigma \la - \rho - \mu)(H)} \boldsymbol{\Gamma}_\mu (\sigma \la) \mathbf{C}_\sigma (\la) v\]
\end{theorem}

The $ \boldsymbol{\Gamma}_\mu $ are given by explicit recursion formulas. Note that in contrast to \eqref{3.5} the order of the factors in each term of the series is important. Also by \cite{H70}, \cite{H73}
\[ \Phi_{\sigma \la, \delta} (x) = \Phi_{\la, \delta} (x) \Gamma_{\sigma, \la}, \]
where $ \Gamma_{\sigma, \la} $ is meromorphic on $ \fa^*_\mathbf{c} $ with values in $ \mathrm{Hom} \, (V_\delta, V_\delta) $ and 
\[ \Gamma_{\sigma, \la} v = \frac{C_{\sigma^{-1}} (\sigma \la)}{\mathbf{c} (\la)} v \quad \text{ for } v \in V^M_\delta. \]

We could also consider the analog of this for the dual space $ \Xi = G/MN. $ For this consider a representation $ \sigma $ of $ MN $ on a finite dimension space $ V_\sigma $ and then take $ V = \text{Hom} \, (V_\sigma, V_\sigma) $ in \S 5 and replace \eqref{1.2} by
\[ \Psi (h \cdot \xi) = \sigma (h) \Psi (\xi) \quad h \in MN, \ \xi \in \Xi,\]
extending $ \Psi $ to distributions. 

This leads to \emph{Whittaker functions} and \emph{Whittaker distributions}, studied for example in Goodman-Wallach [80] where relations with conical distributions is also established. 
 
\section{The asymptotics of $ \Phi_{\la, \delta} $}
In this section we limit ourselves to the case rank $ X= 1. $  Then by Kostant \cite{K69} $ \ell (\delta) = \dim V^M_\delta = 1. $ We fix $ v \in V^M_\delta $ of norm 1 and take $ H \in \fa^+ $ such that $ \al(H) = 1. $ We put $ a_t = \exp t H  $ and 
\begin{equation}
\label{6.1}
\phi _{\la, \delta} (x) = \langle \Phi_{\la, \delta} (x) v, v \rangle.
\end{equation}

Then by Theorem \ref{th5.3}
\[ 
\begin{aligned}
\phi _{\la, \delta} (a_t \cdot \circ) & = e^{(i \la - \rho) (tH)} \sum_{n = 0}^{\infty} e^{-nt} \Gamma_n (\la) \langle \mathbf{C}_e (\la) v, v \rangle \\
& + e^{-(i\la + \rho) (tH)} \sum_{n =0}^{\infty} e^{-nt} \Gamma_n (-\la) \langle \mathbf{C}_\sigma (\la) v, v \rangle 
\end{aligned} \]
Multiply by $ e^{-(i\la - \rho) (tH)}, \quad \la = \xi + i \eta. $

Then if $ \eta < 0 $ we have by \cite{H08}, II, Theorem 3.16, 
\[ \lim\limits_{t \rightarrow \infty} e^{-(i\la - \rho) (tH)} \phi _{\la, \delta} (a_t \cdot \circ) = \mathbf{c} (\la) \langle v, v \rangle. \]

The left hand side is 
\[  \sum_{0}^{\infty} e^{-nt} \Gamma_n (\la) \langle \mathbf{C}_e (\la) v, v \rangle  + e^{-2i\la (tH)} \sum_{0}^{\infty} e^{-nt} \Gamma_n (-\la) \langle \mathbf{C}_\sigma (\la) v,v \rangle \]

Since $ \Gamma_n $ grows at most exponentially $ < \frac{1}{2} n $ (see \cite{H08}, III, \S 5) and $ -2i \la = - 2 i \xi + 2 \eta $ the limit for $ t \rightarrow \infty $ equals $ \mathbf{C}_e (\la) . $ Since both sides are meromorphic in $ \la $ we conclude

\begin{equation}
\label{6.2}
\mathbf{C}_e (\la) = \mathbf{c} (\la).
\end{equation}

We can also multiply the expansion by $ e^{(i \la + \rho)(tH)} $ and for suitable $ \eta  $ deduce
\begin{equation}
\label{6.3}
\lim\limits_{t  \rightarrow \infty} e^{(i \la + \rho)(tH)} \phi_{\la, \delta} (a_t \cdot \circ) = \mathbf{C}_\sigma (\la).
\end{equation}

This in itself does not give much information.

On the other hand, $ \phi_{\la, \delta} $ is an eigenfunction of the Laplacian and by \cite{H76} or \cite{H08}, p. 328, given by 
\begin{equation}
\label{6.4}
\begin{aligned}
& \phi_{\la, \delta} (a_t \cdot \circ)  = c_{\la, \delta} \tanh^s t  \cosh^l t \\
& \!\times\! F\!\left(\thalf (s+r-l), \thalf (s-r-l + 1 -m_{2 \al}), s \!+\! \thalf (m_\al+ m_{2 \al} + 1), \tanh^2 t \right)
\end{aligned}
\end{equation}
where $ l = (i \la - \rho) (H), \ m_\al$ and $ m_{2 \al} $ are the multiplicities of $ \al $ and $ 2 \al, \ r $  and $ s $ are integers, $ r \leq s $ given by 
\[ 
\begin{aligned}
r (r + m_{2 \al} -1) & = - \frac{1}{4} d_{2 \al} \\
s(s+m_\al + m_{2 \al} - 1) & = - d_\al - \frac{1}{4} d_{2 \al} \\
\end{aligned} \]

and $  d_\al $ and $ d_{2 \al} $ determined by
\[ 
\begin{aligned}
\delta (\omega_\al) | V^M_\delta & = d_\al (2 (m_\al + 4 m_{2 \al}))^{-1} \\
\delta (\omega_{2 \al}) | V^M_\delta & = d_{2 \al} (2(m_\al + 4 m_{2 \al}))^{-1}. \\
\end{aligned} \]

The operators $ w_\al $ and $ w_{2 \al} $ are defined in \cite{H76}, \S 4.

Also by \cite{H76}
\begin{equation}
\label{6.5}
c_{\la, \delta} = \frac{\Gamma (\half (\langle i \la + \rho, \al_0 \rangle + s + r))}{\Gamma (\half (\langle i \la + \rho, \al_0 \rangle))} \frac{\Gamma (\half (\langle i \la + \rho, \al_0 \rangle + 1 - m_{2 \al} + s - r))}{\Gamma (\half (\langle i \la + \rho, \al_0 \rangle + 1 - m_{2 \al}))}
\end{equation}

Consider the factor $ (\cosh t)^l, \ l = i \xi (H) - (\eta + \rho) (H) $ and choose $  \eta > 0. $ 

Multiply $ \phi_{\la, \delta} (a_t \cdot \circ) $ by $ (2 \cosh t)^{-\ell} $ and let $ t \rightarrow + \infty.  $ Then 
\begin{equation}
\label{6.6}
\lim\limits_{t \rightarrow \infty} (2 \cosh t)^{-\ell} \phi_{\la, \delta} (a_t \cdot \circ) = c_{\la, \delta} \  \tfrac{1}{2^l} \ \frac{\Gamma (c-a-b) \Gamma (c)}{\Gamma (c-a) \Gamma (c-b)}.
\end{equation}
by the limit formula $ F(a,b,c, \infty), $ where 
\[ a = \thalf (s+r - l), \quad b = \thalf (s-r-l + 1 - m_{2 \al}), \quad c = s+ \thalf (m_\al + m_{2 \al} +1). \]
Then
\[ 
\begin{aligned}
c - a - b & = i \la (H) \\
c-a & = \thalf s - \thalf r + \thalf i \la (H) + \tfrac{1}{4} m_\al + \thalf \\
c - b & = \thalf s + \thalf r + \thalf i \la (H) + \tfrac{1}{4} m_\al + \thalf m_{2 \al} \\
\end{aligned} \]

In \eqref{6.5} recall that $ \al_0 = \al / \langle \al, \al \rangle  $
\[ 
\begin{aligned}
\langle \rho, \al_0 \rangle & = \langle \thalf m_\al \al + m_{2 \al} \al, \al_0 \rangle = \thalf m_\al + m_{2 \al} \\
\langle i \la + \rho, \al_0 \rangle & = i \la (H) + \half m_\al + m_{2 \al} \\
\end{aligned} \]

Thus the right hand side of \eqref{6.6} equals
\[ 
\begin{aligned}
\frac{\Gamma \left( \half \left( i \la (H) + \half m_\al + m_{2 \al} + s + r \right) \right) \Gamma \left( \half \left( i \la (H) + \half m_\al + 1 + s - r \right)\right)}{\Gamma \left( \half \left( i \la (H) + \half m_\al + m_{2 \al}\right) \right) \Gamma \left( \half \left( i \la (H) + \half m_\al + 1\right)\right)} \\
\times \frac{1}{2^l} \,  \frac{\Gamma \left( i \la (H)\right) \Gamma \left( s + \half \left( m_\al + m_{2 \al} + 1 \right)\right)}{\Gamma \left( \half s - \half r + \half i \la (H) + \frac{1}{4} m_\al + \half \right) \Gamma \left( \frac{s}{2} + \frac{r}{2} + \frac{i \la (H)}{2} + \frac{m_\al}{4} + \half m_{2 \al} \right)}
\end{aligned} \]
Here top of the first fraction cancels against the bottom of second fraction. Thus \eqref{6.6} becomes (considering \eqref{3.7}), 
\[ 
\begin{aligned}
& \frac{1}{2^l} \, \frac{\Gamma (i \la (H)) \Gamma \left( s + \half (m_\al + m_{2 \al} + 1) \right)}{\Gamma \left( \half \left( i \la (H) + \half m_\al + m_{2 \al}\right)\right) \Gamma \left( \half  \left( i \la (H) + \half m_\al +1 \right)\right)} \\
& = \frac{\Gamma \left( s + \frac{n}{2}\right)}{\Gamma \left( \frac{n}{2}\right)} \, \mathbf{c} (\la) \hspace{16em} \text{ if } n = \dim X.
\end{aligned} \]


\begin{theorem}
\label{th6.1}
For $ \la = \xi + i \eta, \ \eta > 0,  $

\begin{equation}
\label{6.7}
\lim\limits_{t \rightarrow \infty} (2 \cosh t)^{-\ell} \phi_{\la, \delta} (a_t \cdot \circ) = \frac{\Gamma \left( s + \frac{n}{2}\right)}{\Gamma \left( \frac{n}{2}\right)} \, \mathbf{c} (\la).
\end{equation}
\end{theorem}

So far we have studied the behavior of $ \phi_{\la, \delta} $ for large $ t. $ For behavior for small $ t $ we just use \eqref{6.4} and conclude
\[ \lim\limits_{t \rightarrow 0} \frac{\phi_{\la, \delta} (a_t \cdot \circ)}{\phi_{-\la, \delta} (a_t \cdot \circ)} = \frac{c_{\la, \delta}}{c_{-\la, \delta}}. \]

On the other hand as mentioned in \S 6
\[ \Phi_{s \la, \delta} (x) = \Phi_{\la, \delta} (x) \Gamma_{s, \la}, \]
where 
\[ \Gamma_{s,\la} v = \frac{\mathbf{C}_{s-1} (s\la)}{\mathbf{c}(\la)}v, \hspace{10ex} v \in V^M_{\delta}. \]
In the rank one case where $ \mathbf{C}_\sigma $ is a scalar this implies
\begin{equation}
\label{6.8}
\mathbf{C}_\sigma (-\la) = \frac{c_{-\la, \delta}}{c_{\la, \delta}} \, \mathbf{c} (\la).
\end{equation}

\begin{theorem}
\label{th6.2}
Formulas \eqref{6.2} and \eqref{6.8} determine the $ \mathbf{C} $ functions in the rank one case.
\end{theorem}


\begin{remark}
\label{rem6.3}
By \cite{H73}, Lemma 6.1 and Lemma 6.5 we have
\begin{equation}
\label{6.9}
\mathbf{C}_\sigma (-\la) = \int_N e^{-(i \la + \rho)(H(\bar{n}))} \delta (k(\bar{n})^{-1} m^*)  \, d\bar{n}
\end{equation}
where $ m^* \in W $ is $ -I $ on $ \fa. $ Thus \eqref{6.8} gives an evaluation of the unwieldy integral \eqref{6.9}. In the papers \cite{J76} and \cite{JW72}, \cite{JW77} Johnson and Wallach determined \eqref{6.9} by using a classification of $ \widehat{K}_M $ related to a parametrization from Kostant \cite{K69}. In each case they have a formula in the spirit of \eqref{6.8}. Their models for the $ \delta \in \widehat{K}_M $ are spaces of homogeneous harmonic polynomials on $ \fp, $ restricted by different conditions, according to the multiplicities $ m_\al $ and $ m_{2 \al}. $ This leads to a determinations of the integers s, r for each $ \delta $ in \eqref{6.4} (\cite{H74}, p. 336--337). See also \cite{C74} for the simplest cases. 
\end{remark}

\section{The case of higher rank}
Using a method of Schiffman \cite{S71} we shall now investigate the $ \mathbf{C} $-functions for $ X $ of higher rank. 

As in \cite{H73} we consider the endomorphism $ A (\la, \sigma) $ of $ V^M_\delta $ given by
\begin{equation}
\label{7.1}
A(\la, \sigma)v = \int_{\bar{N}_\sigma} e^{-(i\la + \rho) (H(\bar{n}))} \delta (m_\sigma k (\bar{n})) v \, d \bar{n},
\end{equation}
where $ m_\sigma $ is a representative of $ \sigma \in W $ in the normalizer $ M' $ of $ A $ in $ K. $ Under conjugation $ g \rightarrow mgm^{-1} $ by an $ m \in M $ we have $ (k(\bar{n}))^m = k(\bar{n}^m) $ so $ A (\la, \sigma) $ in \eqref{7.1} is independent of the choice of $ m_\sigma \in M' $ representing $ \sigma \in W. $ Let $ \sigma = \sigma_1 \ldots \sigma_p $ be a reduced expression of $ \sigma,  $ that is each $ \sigma_i $ is a reflection in the plane $ \al_i = 0  $ where $ \al_i $ is a simple root and $ p $ is as small as possible. 

Then by \cite{H73}, following a method by \cite{S71}, 
\begin{equation}
\label{7.2}
A (\la, \sigma) = A (\sigma^{(1)} \la, m_{\sigma_1}) \cdots A (\sigma^{(p)} \la, m_{\sigma_p}),
\end{equation}
where $ \sigma^{(q)} = \sigma_{q+1} \cdots \sigma_p. $

For each simple root $ \al $ consider the rank-one symmetric space $ G_\al / K_\al $ where $ G_\al $ is the analytic subgroup of $ G $ whose Lie algebra is generated by the root spaces $ \fg_\al $  and $ \fg_{-\al} $ and $ K_\al = G_\al \cap K. $ If $ \fa_\al = \mathbf{R}H_\al $ and $ A_\al = \exp \fa_\al, \ \bar{N}_\al = \bar{N}_{s_\al} $ then $ G_\al = K_\al A_\al \bar{N}_\al $ is an Iwasawa decomposition of $ G_\al $. If $ \mathbf{c}_\al $ is the $ \mathbf{c} $ function for $ G_\al / K_\al $ we have by \eqref{7.2} (for $ \delta $ trivial)
\begin{equation}
\label{7.3}
\mathbf{c}_\sigma (\la) = \prod_{1}^{p} \mathbf{c}_{\sigma_j} (\la_j)
\end{equation}
where $ \la_j = (\sigma^{(j)} \la) | \fa_{a_j}. $

\begin{lemma}
\label{lem7.1}
With $ \delta \in \widehat{K}_M $ arbitrary let $ V $ denote the $ K_\al M $ invariant subspace of $ V_\delta $ generated by $ V^M_\delta $ and $ V = \bigoplus\limits^l_{i=1} V_i $ a decomposition into $ K_\al M $-irreducible subspaces. Then $ l = l (\delta), $ the dimension of $ V^M_\delta $ and $ \dim (V^M_\delta \cap V_i) = 1 $ for each $ i. $
\end{lemma}

\begin{proof}
See \cite{H73}, p. 469.
\end{proof}

In this lemma take $ \al = \sigma_j $ and let $ \delta_i $ denote the representation of $ K_{\sigma_j} M $ on $ V_i $ given by $ \delta. $ We choose a unit vector $ v_i $ in $ V_\delta^M \cap V_i. $ Since $ K_{\sigma_j} $ maps $ V^M_\delta \cap V_i $ into itself the operator $ A(\sigma^{(j)}\la, m_{\sigma_j}) $ does too and operates by multiplication with the scalar
\begin{equation}
\label{7.4}
\int_{\bar{N}_{\sigma_j}} \langle \delta_i (m_{\sigma_j} k (\bar{n}) v_i, v_i) \rangle e^{- (i \la_j + \rho_{\sigma_j})(H(\bar{n}))} \, d \bar{n},
\end{equation}
where $ \la_j = (\sigma^{(j)}\la) | \fa_{\sigma_j} $ Here we have used the fact that the restriction $ \rho | \fa_{\sigma_j} $ equals the $ \rho $-function for $ G_{\sigma_j} / K_{\sigma_j} $ \cite{H84} (34) page 446. By \cite{H73}, Lemma 6.5 this number is equal to 
\begin{equation}
\label{7.5}
\int_{\bar{N}_{\sigma_j}} e^{-(i \la_j + \rho_{\sigma_j})(H(\bar{n}))} \langle \delta_i (k(\bar{n})^{-1} m_{\sigma_j})  v_i, v_i \rangle \, d \bar{n},
\end{equation}
which we have calculated in \eqref{6.8}, \eqref{6.9}. Thus the value of \eqref{7.5} equals 
\begin{equation}
\label{7.6}
\frac{c_{-\la_j, \delta_i}}{c_{\la_j, \delta_i}} \mathbf{c}_{\sigma_j} (\la_j).
\end{equation}

The root multiplicities are now the ones in $ G_{\sigma_j} $ and the integers $ r $ and $ s $ are the ones which belong to $ \delta_i. $ Also
\begin{equation}
\label{7.7}
V^M_\delta = \bigoplus\limits_i \mathbf{C} v_i.
\end{equation}

For a fixed $ j $ this represents diagonalization of $ A (\sigma^{(j)}\la, m_{\sigma_j}), \ 1 \leq i \leq l (\delta). $ The determinant of this endomorphism of $ V^M_\delta $ is then 
\begin{equation}
\label{7.8}
\prod_{i = 1}^{l  (\delta)} \left( \frac{c_{-\la_j, \delta_i}}{c_{\la_j, \delta_i}} \mathbf{c}_{\sigma_j} (\la_j) \right).
\end{equation}

Changing to another $ j $ will change the basis $ (v_i) $ giving new representations $ \delta (i,j). $ Consequently, by \eqref{7.3}
\begin{equation}
\label{7.9}
\det (A(\la, \sigma)) |_{V^M_\delta} = \mathbf{c}_\sigma (\la)^{l(\delta)} \prod_{j = 1}^{p} \left( \prod_{i = 1}^{l (\delta)} \frac{c_{-\la_j, \delta(i,j)}}{c_{\la_j, \delta (i,j)}} \right).
\end{equation}

On the other hand, the adjoint of $ A(\bar{\la}, \sigma) $ is given by 
\[ (A(\bar{\la}, \sigma))^* v = \int_{\bar{N}_\sigma} e^{-(-i \la + \rho) (H(\bar{n}))} \delta(k(\bar{n})^{-1} m_\sigma^{-1}) \, d \bar{n} \, v, \]
which by \cite{H73} equals 
\[ \frac{1}{\mathbf{c}(-\la)} \mathbf{c}_\sigma (-\la) \mathbf{C}_{\sigma^{-1}} (-\sigma \la) v. \]
This by \eqref{7.9} determines the determinant of $ \mathbf{C}_\sigma (\la). $

Formula \eqref{7.9} has some resemblance to Theorem 5 in \cite{C74} but it is not clear whether there is a connection. 

We do not deal with the problem of determining $ C_\sigma $ itself but observe that the Hilbert-Schmidt norm is given by \cite{H73}:
\[ ||C_\sigma (\la) ||^2 = |\mathbf{c} (\la)|^2 \, l (\delta). \]


\begin{thebibliography}{9999999}

\bibitem[BM60]{BM60} 
T. S. Bhanu Murti,  Plancherel's measure for the factor-space $ SL(n;R)/SO(n;R) $. \textit{Dokl. Akad. Nauk SSSR} \textbf{133} 503--506.

T. S. Bhanu Murti,  The asymptotic behavior of zonal spherical functions on the Siegel upper half-plane. \textit{Dokl. Akad. Nauk SSSR} \textbf{135} 1027--1030.

\bibitem[C74]{C74}
L. Cohn,
\textit{Analytic Theory of Harish-Chandra's c-functions}. Lecture Notes in Math. \textbf{428}, Springer-Verlag, 1974.

\bibitem[G52]{G52}
R. Godement, A theory of spherical functions.\textit{ Trans. Amer. Math. Soc.} \textbf{73} (1952), 496--556.

\bibitem[GK62]{GK62}
S. G. Gindikin and F. I. Karpelevic, Plancherel measure for symmetric Riemannian spaces of non-positive curvature. \textit{Dokl. Akad. Nauk SSSR} \textbf{145} (1962), 252--255.

\bibitem[GW80]{GW80}
R. Goodman and N. Wallach, Whittaker vectors and conical vectors. \textit{J. Funct. Anal.} \textbf{39} (1980), 199--279.

\bibitem[HC54]{HC54}
Harish-Chandra, 
 Representations of semisimple Lie groups. II. \textit{Trans. Amer. Math. Soc.} \textbf{76}, (1954), 26--65. 


\bibitem[HC57]{HC57}
Harish-Chandra,
Differential operators on a semisimple Lie algebra
\textit{Amer. J. Math.} \textbf{79} (1957), 241--310.

\bibitem[HC58]{HC58}
Harish-Chandra,
Spherical functions on a semisimple Lie group I, II
\textit{Amer. J. Math} \textbf{80} (1958), 241--310, 553--613.


\bibitem[HC72]{HC72}
Harish-Chandra, On the theory of the Eisenstein integral. Lecture Notes in Mathematics, \textbf{266} (1972), 123--149. Springer-Verlag. 

\bibitem[H62]{H62}
S. Helgason,
\textit{Differential Geometry and Symmetric Spaces.} Academic Press, NY, 1962.

\bibitem[H65]{H65}
S. Helgason, Radon-Fourier transforms on symmetric spaces and related group representations. \textit{Bull. Amer. Math. Soc. } \textbf{71} (1965), 757–-763.

\bibitem[H70]{H70} 
S. Helgason, 
A duality for symmetric spaces with applications to group representations.
\textit{Adv. Math.} \textbf{5} (1970), 1--154.

\bibitem[H73]{H73} 
S. Helgason,
The surjectivity of invariant differential operators on symmetric spaces,
Ann. of Math. \textbf{98} (1973), 451--480.

\bibitem[H74]{H74}
S. Helgason,
Eigenspaces of the Lapacian. Integral representations and irreducibility. \textit{J. Funct. Anal.} \textbf{17} (1974), 328--353.

\bibitem[H76]{H76}
S. Helgason, A duality for symmetric spaces with applications to group representations. II. Differential equations and eigenspace representations. \textit{Advances in Math.} \textbf{22} (1976), no. 2, 187--219.

\bibitem[H84]{H84}
S. Helgason,
\emph{Groups and Geometric Analysis. Integral Geometry, Invariant differential operators and Spherical functions} Academic Press, NY, 1984. 

\bibitem[H00]{H00}
S. Helgason, Harish-Chandra's $ c $-function. A mathematical jewel. In ``The mathematical legacy of Harish-Chandra'' (Baltimore, MD, 1998), 43–45, \textit{Proc. Sympos. Pure Math.}, \textbf{68}, Amer. Math. Soc., Providence, RI, 2000.

\bibitem[H08]{H08}
S. Helgason,
\textit{Geometric Analysis on Symmetric Spaces}. Math Surveys and Monographs No. \textbf{39}. Amer. Math. Soc. Providence, RI. Second edition, 2008. 

\bibitem[J76]{J76}
K. Johnson,
Composition series and intertwining operators for the spherical principal series II. \textit{Trans. Amer. Math. Soc.} \textbf{215} (1976), 269--283.

\bibitem[JW72]{JW72}
K. Johnson and N. Wallach,
Composition series and intertwining operators for the spherical principal series. \textit{Bull. Amer. Math. Soc.} \textbf{78} (1972), 1053--1059.

\bibitem[JW77]{JW77}
K. Johnson and N. Wallach,
Composition series and intertwining operators for the spherical principal series I. \textit{Trans. Amer. Math.} \textbf{229} (1977), 137--173. 

\bibitem[K69]{K69}
B. Kostant,
On the existence and irreducibility of certain series of representations. \textit{Bull. Amer. Math. Soc.} \textbf{75} (1969), 627--642. 

\bibitem[S71]{S71}
B. Schiffman,
Integrales d'entrelacement et fonctions de Whittaker, \textit{Bull. Soc. Math. France} \textbf{99} (1971), 3--72.

\bibitem[T76]{T76}
J. Tirao, 
Spherical function. \textit{Rev. UN Mat. Argentina} \textbf{28} (1976/77), 75--98.

\bibitem[W75]{W75}
N. Wallach,
On Harish-Chandra's generalized $ c $-functions. \textit{Amer. J. Math.} \textbf{97} (1975), 986--403.



%






\end{thebibliography}
\end{document}